\documentclass[11pt]{amsart}

\usepackage{multicol}
\usepackage{enumerate}
\usepackage{amsthm,amssymb,amscd,amsmath,mathrsfs} 
\usepackage[T1]{fontenc}
\usepackage{ae,aecompl}
\usepackage{microtype}
\usepackage{color}
\usepackage[UKenglish]{babel}
\usepackage{times}
\usepackage[all]{xy}
\usepackage[latin1]{inputenc}

\newtheorem{theorem}{Theorem}

\newtheorem{corollary}[theorem]{Corollary}
\newtheorem{criterion}[theorem]{Criterion}
\newtheorem{definition}[theorem]{Definition}

\newtheorem{lemma}[theorem]{Lemma}

\newtheorem{proposition}[theorem]{Proposition}
\newtheorem{remark}[theorem]{Remark}

\newcommand{\pp}{\mathbb{P}}
\newcommand{\OO}{{\mathcal O}}

\newcommand{\Pic}{\operatorname{Pic}}

\newcommand{\rk}{\operatorname{rk}}

\newcommand{\im}{\operatorname{im}}
\newcommand{\codim}{\operatorname{codim}}

\newcommand{\degl}{\deg_L}
\newcommand{\mc}[1]{\mathcal{#1}}
\newcommand{\e}[1]{\mathbf{e}^{#1}}

\newcommand{\Z}{\mathbb{Z}}
\newcommand{\rG}{\rm G}

\newcommand{\la}{\lambda}

\newcommand{\qand}{\quad\text{and}\quad}

\usepackage[normalem]{ulem}

\begin{document}

\title{Holomorphic bundles for higher dimensional gauge theory}

\author{Marcos Jardim, Gr\'egoire Menet, Daniela M. Prata \and Henrique~N.~S\'a~Earp}
\address{IMECC - UNICAMP \newline
Rua S\'ergio Buarque de Holanda 651 \newline 
Cidade Universit\'aria, 13083-859, Campinas-SP, Brazil}
\email{jardim@ime.unicamp.br, menet@ime.unicamp.br,\newline daniela.mprata@gmail.com, henrique.saearp@ime.unicamp.br}

\begin{abstract}
Motivated by gauge theory under special holonomy, we present techniques to produce holomorphic bundles over certain noncompact $3-$folds, called building blocks, satisfying a stability condition `at infinity'. Such bundles  are known to parametrise solutions of the Yang-Mills equation over the \linebreak $\rG_2-$manifolds obtained from asymptotically cylindrical Calabi-Yau $3-$folds studied by Kovalev, Haskins et al. and Corti et al..

The most important tool is a generalisation of  Hoppe's stability criterion to holomorphic bundles over smooth projective varieties $X$ with $\Pic{X}\simeq\mathbb{Z}^l$, a result which may be of independent interest.

Finally, we apply monads to produce a prototypical model of the curvature blow-up phenomenon along a sequence of asymptotically stable bundles  degenerating into a torsion-free sheaf.

\end{abstract}

\maketitle

\vspace{-0.6cm}
\tableofcontents

\newpage
\section{Introduction}

This paper presents cohomological methods to construct and study examples of holomorphic bundles, over certain Fano $3-$folds, having the property of \emph{asymptotic stability} (see below) on a distinguished anticanonical divisor, said to be `at infinity' for geometrical reasons. It is based on the theory of instanton monads developed by several authors in the past three decades, in particular  \cite{Jardim2006,Jardim2007}.
Our motivation comes from gauge theory in higher dimensions, especially the interplay between Yang-Mills theory over Calabi-Yau 3-folds and $\rG_2-$manifolds.

An important method to produce examples of compact \linebreak 7--manifolds with holonomy exactly $\rG_2$ is the \emph{twisted connected sum} construction  \cite{Kovalev2003,Kovalev2011,Corti2012}. It consists of gluing a pair of asymptotically cylindrical (ACyl) Calabi--Yau $3$--folds obtained from certain  smooth projective $3-$folds called  \emph{building blocks}. A building block $(Z,D)$ \cite[Definition~3.5]{Corti2012}
is given by a projective morphism $f: Z\to \pp^1$ such that $D:=f^{-1}(\infty)$ is a smooth anticanonical  $K3$ surface, under certain  mild topological assumptions; in particular, $D$ has trivial normal bundle. Choosing a convenient Kähler structure on $Z$, one can make $W:=Z\setminus D$ into an ACyl Calabi--Yau $3$--fold, that is, a non-compact Calabi--Yau with a tubular end modelled on $\mathbb{R}_+\times S^1\times D$. Then $S^1\times W$ is an ACyl $\rG_2$--manifold with a tubular end modelled on $\mathbb{R}_+\times T^2\times D$ \cite{Haskins2015,Corti2012a}.

In particular, several topological types of building blocks can be obtained as $Z=\operatorname{Bl}_\ell X$ by blowing up  Fano $3-$folds $X$ along a self-intersection curve $\ell\in \left\vert D\cdot D\right\vert$ for $D\in\left\vert-K_X\right\vert$; these will be the varieties of interest in the present paper. It is reasonable to expect similar methods to work on more general building blocks (e.g. blow-ups of so-called \emph{weak} Fanos), which shall be addressed  in the future. For a more detailed exposition of building blocks, we suggest the Introduction of \cite{SaEarp2013} and references therein.

From the point of view of gauge theory, the last named author established the existence of Hermitian Yang-Mills (HYM) connections over ACyl Calabi-Yau 3-folds \cite{SaEarp2009,SaEarp2011}. The concept of asymptotic stability emerges as the natural boundary condition for that analytical problem. Let $z=\e{-s+\bf{i}\alpha} $ be
the holomorphic coordinate along the tubular end, and denote by $D_z$ the corresponding
$K3$ fibre near infinity.
\label{def bundle E->W}
\index{infinity!stable at}
A bundle $E\rightarrow W$ is called \emph{asymptotically stable}
(or \emph{stable at infinity}) if it is the
restriction of an indecomposable holomorphic vector bundle $E\rightarrow
Z$ such that
        $\left. E\right\vert _{D}$ is stable (hence also $\left.
E\right\vert _{D_{z}}$ for small  $\left\vert z\right\vert <\delta
$).
Such a bundle admits a smooth Hermitian \emph{reference metric} $H_{0}$, with
the property that $\left. H_{0}\right\vert _{D_{z}}$ are the corresponding
HYM metrics on $\left. E\right\vert _{D_z}$, for $0\leq \left\vert
z\right\vert <\delta$, and which has `finite energy', in a suitable sense.

Our crucial motivation is the fact that, given an asymptotically stable bundle
with reference metric $\left(E,H_0  \right)$, there exists a nontrivial
smooth solution to the $\rG_2-$instanton equation on $p_1^*E\rightarrow
W\times S^1$ \cite[Theorem 58]{SaEarp2011}. Our central aim therefore is to construct explicit examples of such asymptotically stable bundles; actually it suffices to construct bundles $E\to X$ over the original Fano 3--fold with $E|_{D}$ stable, since stability pulls back under the proper transform $Z\to X$. 

It should be noted that, under certain rigidity and transversality assumptions, such solutions can be glued, according to the
twisted connected sum, to produce a  $\rG_2-$instanton
over the resulting \emph{compact} 7-manifold with holonomy $\rG_2$ 
\cite{SaEarp2013}.
Thus transversal matching pairs of asymptotically stable bundles parametrise \linebreak (some) solutions to the corresponding  $7-$dimensional
Yang-Mills equation. However, the matching problem
for the tubular $\rG_2-$instantons obtained with our methods is a nontrivial matter to be addressed in future work.






\subsection*{Outline}
This article is organized as follows.  In
\emph{Section \ref{sec: Hoppe criterion}}
 we review standard stability theory and we establish  a generalisation of the so-called \emph{Hoppe criterion},
which gives sufficient conditions for stability of a bundle over a projective variety
with finitely generated Picard group in terms of the vanishing of certain cohomologies. In \emph{Section \ref{asb}} we construct bundles over  various types of building blocks of Picard rank $1$. In \emph{Section  \ref{sec: Hartshorne--Serre}} we use the famous Hartshorne--Serre correspondence to construct examples over polycyclic Fano 3-folds, and apply the generalised stability criterion to establish their 
asymptotic stability.  Finally,
\emph{Section \ref{sec: degenerations}}
 has a somewhat
different vein, illustrating a convenient use  of instanton monads to model degenerations of asymptotically
stable bundles into torsion-free sheaves, with an explicit calculation of
the curvature blow-up rate for a natural choice of metrics.

\subsection*{Acknowledgements}
MJ is partially supported by  CNPq grant 303332/2014-0 and FAPESP grant 2014/14743-8. GM is supported by FAPESP grant 2014/05733-9. DMP was supported by FAPESP grant 2011/21398-7. HSE was supported by FAPESP grant 2009/10067-0 and is partially supported by CNPq grant 312390/2014-9 and FAPESP grant 2014/24727-0. We thank A. A. Henni and P. Coronica for comments and T. Walpuski for suggesting the topic of  \emph{Section \ref{sec: degenerations}}.
The authors thank especially the referee for numerous contributions to the manuscript. 

\newpage
\section{The generalised Hoppe criterion}
\label{sec: Hoppe criterion}

Let $X$ be a nonsingular projective variety, i.e. a nonsingular, projective, integral, separated Noetherian scheme of finite type over the field of complex numbers. Let $L$ be an ample line bundle over $X$, and set $n:=\dim X$.

The  \emph{$L-$degree} of a coherent sheaf $E\to X$ is defined as usual by
$$\degl E:= c_1(E)\cdot L^{n-1},$$
and, setting $r:=\rk(E)$, the \emph{$L$--slope} of $E$ is
\begin{equation}        \label{eq: mu_L}
        \mu_L( E) := 
        \frac{\degl E}{r} .
\end{equation}
Then $E $ is \emph{(semi)stable} if, for every proper coherent subsheaf $F \subset E$ such that $E/F$ is torsion-free, one has
$$
\mu_L (F) \underset{(\leq)}{<} \mu_L (E).
$$
If $E$ is locally free, in order to test for stability it suffices to consider  \emph{reflexive} subsheaves $F\subset E$.

\subsection{Hoppe's criterion over cyclic varieties}
Suppose further that $\Pic(X)\simeq\mathbb{Z}$; such varieties are called \emph{cyclic}. Given a locally free sheaf (or, equivalently, a holomorphic vector bundle) $E\to X$ as above, there is a unique integer $k_E$ such that 
$$
-r+1\le c_1(E(-k_E))\le 0.
$$ 
Setting $E_{\rm norm}:=E(-k_E)$, we say $E$ is \emph{normalized} if $E=E_{\rm norm}$. Then one has the following stability criterion 
\cite[Lemma 2.6]{Hoppe1984}:
\begin{proposition}[Hoppe criterion]\label{prop hoppe}
Let $E$ be a rank $r$ holomorphic vector bundle over a cyclic projective variety $X$. If $H^0((\wedge ^qE)_{\rm norm})=0$ for $1\leq q\leq r-1$, then $E$ is stable. If $H^0((\wedge ^qE)_{\rm norm}(-1))=0$ for $1\leq
q\leq r-1$, then $E$ is semistable.
\end{proposition}

\noindent NB.: In general, the converse of the previous statement is false. Take for instance the nullcorrelation bundle  $N\to\pp^{2k+1}$ with $k\ge 2$, i.e., the stable rank $2k$  bundle  \cite[Thm 1.4]{Ein1982} given by the following exact sequence:
$$ 
\xymatrix{
0 \ar[r]& \OO_{\pp^{2k+1}}(-1) \ar[r]& \Omega^1_{\pp^{2k+1}}(1) \ar[r]& N \ar[r]& 0. 
}
$$
As shown in \cite[Lemma 1.10]{Ancona1994},  for each $0\le j\le k-1$, the exterior product  $\Lambda^{2j}N$ has the trivial line bundle $\OO_{\pp^{2k+1}}$ as a direct summand, hence a nontrivial section.

Hoppe's original result is proved only for the case $X=\pp^n$, but the proof generalises easily for cyclic varieties. Conversely, for rank $2$ bundles stability is in fact equivalent to the nonexistence of holomorphic sections:

\begin{criterion}\label{crit hoppe}
If $F$ is a rank $2$ holomorphic vector bundle over a cyclic variety $X$, then $F$ is stable if and only if $h^0(F_{\rm norm})=0$.
\end{criterion}

\noindent Compare the above with \cite[Lemma 1.2.5, p.165]{Okonek1980} for the case $X=\pp^n$; the proof for $X$ cyclic is the same. A similar necessary and sufficient criterion for rank $3$ bundles can also be found in \cite[Lemma 1.2.6b, p.167]{Okonek1980}. In any case, Hoppe's approach has been, up until now, the most general framework for cohomological stability criteria.

\subsection{Hoppe's criterion over polycyclic varieties}


We now present a generalisation of the Hoppe stability criterion for the
much larger class of \emph{polycyclic} varieties with $\Pic(X)=\Z^\ell$.
Given a divisor $B$ on $X$, we define, for convenience, $$\delta_L (B):=\degl\OO_X(B). $$

\begin{theorem}[Generalized Hoppe Criterion]\label{thm: GHC}
Let  $G \rightarrow X$ be a holomorphic vector bundle of rank $r \geq 2$ over a polycyclic variety $X$ equipped with a polarisation $L$; if
$$
H^{0}(X, (\wedge^s G)\otimes\OO_X(B)) = 0 
$$
for all $B \in \Pic(X)$ and all $s\in\left\{1,\dots,(r-1)\right\}$ such that
$$
\delta_L(B)  \underset{(<)}{\leq} -s\mu_{L}(G) 
$$
then $G$ is (semi-)stable.

Conversely, if $G$ is (semi-)stable then
$$
H^{0}(X, G\otimes\OO_X(B)) = 0 ,\; \forall  B \in \Pic(X) 
\,\text{ such that } \delta_L(B) \underset{(<)}{\leq} - \mu_L (G).
$$
\end{theorem}

\begin{proof}
First assume that there is $F \hookrightarrow G$ a torsion free sub-sheaf of rank $s$, such that $\det F = \OO_ X(B)$. The inclusion induces a map $ \wedge^s F \hookrightarrow \wedge^s G$ and so 
$$
H^{0}(X, (\wedge^sG)\otimes\OO_X(-B)) \neq 0.
$$
By hypothesis, we have
$$
\delta_L(-B)  \underset{(\geq)}{>} -s\mu_{L}(G),
$$
hence
$$
\mu_{L}(F)=\frac{\delta_L(B)}{s}  
\underset{(\leq)}{<} \mu_{L}(G)
$$
and therefore $F$ is not a destabilising sheaf. 

Conversely, assume that $H^{0}(X, G\otimes\OO_X(B))$ has a section with 
$$
\delta_L(B) \underset{(<)}{\leq} -\mu_L(G).
$$
This induces a map $\OO_X(B) \hookrightarrow G$ with
$$
\mu_L(\OO_X(B)) \underset{(>)}{\geq} \mu_L(G)\\
$$
so $G$ is not (semi-)stable.
\end{proof}

\begin{corollary}
Let  $G \rightarrow X$ be a holomorphic vector bundle of rank 2 over a polycyclic variety, and let $L$ be a polarisation on $X$. The bundle $G$ is (semi)-stable if and only if 
$$
H^{0}(X, G\otimes\OO_X(B)) = 0
$$
for every $B \in \Pic(X)$ such that 
$$
\delta_L(B)  \underset{(<)}{\leq} -\mu_{L}(G). 
$$
\end{corollary}

\section{Asymptotic stability over cyclic Fano-type building blocks}\label{asb}

The existence of asymptotically stable bundles is a natural and important question, since they parametrise solutions to the HYM equation and hence \linebreak $\rG_2-$instantons.
Let us observe form the outset that an affirmative example, if somewhat exotic, is known from Mukai's study of certain prime subvarieties $X_{22}$ $\subset\pp^{13}$; such spaces come equipped with a rank $2$ holomorphic bundle which is indeed asymptotically stable \cite{Mukai1987,Mukai1992}.
We propose a much more general setup to obtain examples over a larger class of building blocks. Let us begin by recalling their precise definition  \cite[Definition~3.5]{Corti2012}. 

\begin{definition}
A \emph{building block} is a nonsingular algebraic 3-fold $Z$ together with a projective morphism $f: Z\rightarrow \mathbb{P}^{1}$ satisfying the following assumptions:
\begin{itemize}
\item[(i)] 
the anti-canonical class $-K_{Z}\in H^{2}(Z,\Z)$ is primitive.
\item[(ii)]
$D=f^{-1}(\infty)$ is a non-singular K3 surface and $D\sim -K_{Z}$.
\end{itemize}
In addition, identify $H^{2}(D,\Z)$ with the K3 lattice $L$ (i.e. choose a marking for $D$), and let $N$ denote the image of $H^{2}(Z,\Z)\rightarrow H^{2}(D,\Z)$.
\begin{itemize}
\item[(iii)]
The inclusion $N\hookrightarrow L$ is primitive.
\item[(iv)]
The groups $H^{3}(Z,\Z)$ and $H^{4}(Z,\Z)$ are torsion-free.
\end{itemize}
\end{definition}

The building blocks considered in this paper are given by \cite[Definition~3.15]{Corti2012}.

\begin{proposition}\label{FanoBlock}
Let $X$ be a Fano 3-fold, $|D_{0},D_{\infty}|\subset |-K_{X}|$ a generic pencil with (smooth) base locus $\ell$, $D\in |D_{0},D_{\infty}|$ generic, and $Z$ the blow-up of $X$ at $\ell$. Then $D$ is a smooth K3 surface, its proper transform in $Z$ is isomorphic to $D$, and $(Z,D)$ is a building block. Such building blocks are called of \emph{Fano type}, while $X$ is refered to as the \emph{underlying Fano 3-fold}.
\end{proposition}

Throughout this Section, $X$ will denote the Fano 3-fold underlying a building block $(Z,D)$ constructed as above.

\begin{remark}
Let $(Z,D)$ be a building block, then we say that a bundle $E$ on $Z$ is \emph{asymptotically stable} if $E_{|D}$ is stable. Similarly, if $X$ is a Fano 3-fold, $F$ a bundle on $X$ and $D\in |-K_{X}|$ a smooth K3 surface, we say that $F$ is \emph{asymptotically stable} if $F_{|D}$ is stable.

We remark that in order to provide asymptotically stable bundles on a given building block it is enough to construct asymptotically stable bundles on its underlying Fano 3-fold. Indeed, let $X$ be a Fano 3-fold, $|D_{0},D_{\infty}|\subset |-K_{X}|$ a generic pencil with (smooth) base locus $\ell$, $D\in |D_{0},D_{\infty}|$ generic, and $r:Z\rightarrow X$ the blow-up of $X$ at $\ell$. If $F$ is an asymptotically stable bundle on $X$ then $r^{*}F$ is an asymptotically stable bundle on $Z$. This is the strategy we adopt below.
\end{remark}

\subsection{Stable bundles from monad constructions}
\label{subsec: stable bundles from monads}

Let us briefly review one of the main techniques to produce holomorphic bundles
with prescribes topology and stability properties. This technology
will then be mildly adapted for our main purpose of constructing asymptotically
stable bundles. 

 A \emph{monad} on $X$ is a complex of locally free sheaves
$$
\xymatrix{
{\rm M}_\bullet : & M_{0} \ar[r]^-{\alpha} & M_1 \ar[r]^-{\beta} & M_2 }
$$
such that $\beta$ is locally right-invertible and $\alpha$ is locally left-invertible. The (locally free) sheaf $E:=\ker\beta/\im\alpha$ is called the \emph{cohomology} of ${\rm M}_\bullet$.
Monads are a valuable tool in the theory of sheaves over projective varieties and have been studied by many authors over the past four decades. We will be fundamentally interested in the so-called \emph{linear sheaves} on $X$, i.e., coherent sheaves (of rank $r$) which can be obtained as the cohomology of a \emph{linear monad} of the form
\begin{equation}        \label{eq: linear monad}
\xymatrix{
0 \ar[r] & \OO_{X}(-1)^{\oplus c} \ar[r]^-{\alpha } & \OO_{X}^{\oplus r+2c} \ar[r]^-{\beta} & \OO_{X}(1)^{\oplus c} \ar[r] & 0.
}
\end{equation}

The paper \cite{Jardim2007} is of particular interest to us, since the first-named author constructed examples of stable linear sheaves over cyclic varieties of dimension $3$.
More precisely, the following result is proved, as an application of the Hoppe criterion (Proposition \ref{prop hoppe}):

\begin{theorem} \label{Thm cyclic instanton monad }
Let $X$ be a cyclic nonsingular complex projective 3-fold with fundamental class $h:= c_1(\OO_{X}(1))$, and $c\geq1$ an integer; then a linear monad of the form
\begin{equation} \label{eq instanton monad}
\xymatrix{
0 \ar[r] & \OO_{X}\left(-1\right)^{\oplus c} \ar[r]^-\alpha & \OO_{X}^{\oplus2+2c} \ar[r]^-\beta & \OO_X\left(1\right)^{\oplus c} \ar[r]&  0
}
\end{equation}
has the following properties:
\begin{itemize}
\item[(i)] the kernel $K:=\ker\beta$ is a stable bundle with
\begin{displaymath}
\rk(K)=c+2,\quad c_1(K)=-c . h, \quad c_2(K)=\frac{1}{2}\left( c^2+c \right). h^2.
\end{displaymath}
\item[(ii)] the linear sheaf  $E:= \ker\beta/\im \alpha$ is a stable bundle with
\begin{displaymath}
\rk(E)=2,\quad c_1(E)=0, \quad c_2(E)=c\cdot h^2.
\end{displaymath}
\end{itemize}
\end{theorem}

A simple example of a linear monad of the form (\ref{eq instanton monad}) with $c=1$ can be given as follows. 
Let $X$ be a nonsingular hypersurface of degree $d$ in $\mathbb{P}^4$ not containing the point $[0:0:0:0:1]$. Then the complex
\begin{equation}\label{ex-monad}
\xymatrix{
0 \ar[r]& \OO_{X}(-1) \ar[r]^-\alpha &
\OO_{X}^{\oplus 4} \ar[r]^-\beta& \OO_{X}(1) \ar[r]& 0
}\end{equation}
given in homogeneous coordinates $[x_0:x_1:x_2:x_3:x_4]$ of $\mathbb{P}^4$ by 
$$ \alpha = \left(\begin{array}{c} x_3 \\ x_2 \\ -x_1 \\ -x_0 \end{array}\right)
~~ {\text{and}} ~~
\beta = \left(\begin{array}{cccc} x_0 ~&~ x_1 ~&~ x_2 ~&~ x_3 \end{array}\right) ~~ $$
is a monad on $X$, since $\alpha$ is injective and $\beta$ is surjective at every point of $X$. Examples with higher values for $c$ may be found in \cite[Section 3]{Jardim2007}.

\subsection{Asymptotically stable monad cohomologies}

Let $(Z,D)$ be a building block of Fano type, with $X$ being the underlying Fano 3-fold. Suppose  further that $X$ is  cyclic and that  $D\in\left\vert -K_{X} \right\vert$ is a generic cyclic anticanonical divisor, so that 
$$ \Pic X\simeq\Z\simeq\Pic D. $$ 
Then we can assign a monad construction with stable cohomology bundle $E$, say, to a large number of such underlying Fano 3-folds, for instance:
\begin{multicols}{2}
\begin{enumerate}[(i)]
\label{list Kovalev}
        \item
        $X=\mathbb{P}^3$;
        \item
        $X \overset{2:1}{\underset{D}{\longrightarrow}}\mathbb{P}^3$;
        \item
        $X\subset \mathbb{P}^4$, hypersurface of degree 2, 3 or 4;
\end{enumerate}
\end{multicols}

The strategy consists in adapting the linear monads from Theorem \ref{Thm cyclic instanton monad } to obtain instanton bundles over $X$. However, when generalising results from projective spaces to wider classes of projective varieties, one often encounters difficulties because $\pp^n$ does not `have as much cohomology'. This principle has the following precise significance for us:      
\begin{definition}\label{def-wic}
A line bundle $L\rightarrow X$ over a projective variety of dimension $n$ is said to be \emph{without intermediate cohomology (WIC)} if
\begin{eqnarray*}
        H^i\left( L^{\otimes k}\right)=0,
        \qquad \forall
        \left\{
        \begin{array}{l}
                i=1,\dots,n-1 \\
                k\in\mathbb{Z} \\
        \end{array}.
        \right.
\end{eqnarray*}
A projective variety $X$ is said to be WIC if the line bundle $\OO_X(1)$ is WIC.
\end{definition}

\begin{remark}   \label{lemma Ox(1) is WIC}
It is not difficult to see, using Kodaira's vanishing theorem, that if $X$ is cyclic and Fano, then the positive generator of $\Pic X$, denoted $\OO_X(1)$, is WIC. 
\end{remark}
\noindent Note that complete intersection subvarieties of dimension at least $3$ in $\pp^n$, $n\geq4$, are cyclic and WIC.

Then indeed our instanton linear sheaves are  asymptotically stable bundles:

\begin{proposition}  \label{Prop asymp stable middle cohomology}
Let $X$ be a nonsingular, cyclic Fano threefold, and let $D\subset X$ be a cyclic anticanonical divisor. If $E\rightarrow X$ arises from an instanton monad of the form $(\ref{eq instanton monad})$, then $E$ is an asymptotically stable bundle.
\end{proposition}

The rest of this \emph{Subsection} is devoted to the proof of Proposition \ref{Prop asymp stable middle cohomology}. First of all,  since $E$ can be obtained as the cohomology of a monad of the form $(\ref{eq instanton monad})$, Theorem~\ref{Thm cyclic instanton monad } guarantees that it is stable. Denote by   $\OO_X(1)$ the polarisation,   $\sigma_D\in H^0(K^{-1}_X)$  the section cutting out $D$,  $d:=\deg D$ its degree, and $r_{X,D}$ the restriction map; so the restriction sequence reads
\begin{equation}        \label{eq restriction sequence}
\xymatrix{
0 \ar[r]&E(-d) \ar[r]^-{\sigma_D}& E \ar[r]^-{r_{X,D}}&E|_D \ar[r]&0
}
\end{equation}
Moreover, setting $K:=\ker\beta$ and twisting the monad by $\OO_X(-d)$, the relevant data fit in the following \emph{canonical diagram}:

$$
\xymatrix{
&0 \ar[d]&&&\\
& \OO_X\left(-(d+1)\right)^{\oplus c} \ar[d] & & & \\
0 \ar[r] & K(-d) \ar[r]\ar[d] & \OO_X(-d)^{\oplus 2+2c} \ar[r] & \OO_X(-\left(d-1\right))^{\oplus c} \ar[r] & 0\\
0 \ar[r] & E(-d) \ar[r]\ar[d] & E \ar[r] & E|_D \ar[r] & 0\\
&0&&&
}
$$

Note that $\left.E\right \vert_D$ is a rank $2$ bundle with $c_1(\left.E\right \vert_D) = 0$, thus $\left.E \right \vert_D = \left. (E \right \vert_D)_{\rm norm}$. In view of Criterion \ref{crit hoppe} and since $h^0(E)=0$, it follows from the second row of the canonical diagram that it suffices to check the vanishing at
\begin{displaymath}
        h^1\left( E(-d) \right)=0.
\end{displaymath}
Note that $h^0(\OO_X(-k))=0$ and, since $X$ is WIC (cf. Remark \ref{lemma Ox(1) is WIC}), $h^2\left( \OO_X\left( k \right) \right)=0$ for all $k\in\mathbb{Z}$. It follows from the first row in the canonical diagram that $h^1(K(-d))=0$. Finally, from the column of the canonical diagram we have $h^1\left( E(-d) \right)=0$, which concludes the proof.

In particular, Proposition \ref{Prop asymp stable middle cohomology} gives many examples of rank $2$ asymptotically stable bundles over varieties such as (i) and (iii) on page \pageref{list Kovalev}. The same examples can also be pulled back to double covers of type (ii).

\section{Asymptotically stable Harshorne-Serre bundles}
\label{sec: Hartshorne--Serre}

Let $X$ be a complex manifold and $Y\subset X$ be a codimension 2 local complete intersection subscheme. A bundle $E\to X$ of rank $r$ will be called a \emph{Hartshorne--Serre bundle obtained from $Y$} if 
there exists a line bundle $\mathcal{L}$ with an exact sequence: 
$$\xymatrix{ 0\ar[r]&\mathcal{O}_{X}\ar[r] & E\ar[r] & \mathcal{I}_{Y}\otimes \mathcal{L}\ar[r]& 0,}$$
where $I_{Y}$ is the ideal sheaf of $Y$ in $X$. Heuristically, we think of $E$ as a global extension of the normal bundle of $Y$, in a sense which we will now make precise for the case of rank $2$.

The following instance of  \cite[Theorem 1]{Arrondo} gives sufficient conditions for the existence of Hartshorne--Serre bundles; to the interested reader, we strongly recommend the provided reference for a detailed and user-friendly exposition of that  construction.

\begin{theorem}[Hartshorne--Serre construction in rank $2$]
\label{thm: Hartshorne--Serre}
Let $Y\subset X$ be a local complete intersection subscheme of codimension 2 in a smooth algebraic variety. If there exists a line bundle $\mathcal{L}\to X$ such that $H^2(X,\mathcal{L}^*)=0$, and $\wedge^2\mathcal{N}_{Y/X}=\mathcal{L}|_Y$, then there exists a rank $2$ Hartshorne--Serre bundle obtained from $Y$ such that
\begin{itemize}
\item[(i)] $\wedge^2E=\mathcal{L}$,
\item[(ii)] $E$ has a global section whose vanishing locus is $Y$.
\end{itemize}
\end{theorem}

\subsection{Generalized Hoppe criterion for the Harshorne-Serre construction}

We will produce asymptotically stable examples over certain Fano-type building blocks as Hartshorne--Serre bundles. In order to check their stability, the following adapted version of the polycyclic criterion of Theorem \ref{thm: GHC} will be instrumental:
\begin{proposition}\label{Hartshorne}
Let $X$ be a polycyclic complex manifold  endowed with a polarisation $L$. Let $E$ be a rank $2$ Hartshorne--Serre bundle obtained from a codimension $2$ complete intersection sub-scheme $Y$.
Then $E$ is stable (resp. semi-stable) if 
\begin{itemize}
\item[(i)] $\mu_{L}(E)\underset{(\geq)}{>}0$, 
\item[(ii)] for every effective divisor $\mathcal{S}$ with $\delta_L(\mathcal{S})  \underset{(<)}{\leq} \mu_{L}(E)$, $Y$ is not contained in $\mathcal{S}$.
\end{itemize}
\end{proposition}

\begin{proof}
We have the following exact sequence:
$$\xymatrix{ 0\ar[r]&\mathcal{O}_{X}\ar[r] & E\ar[r] & \mathcal{I}_{Y}\otimes \mathcal{L}\ar[r]& 0.}$$
To apply Theorem \ref{thm: GHC}, we tensorize by $\mathcal{O}_{X}(B)$ with $\delta_L(B)  \underset{(<)}{\leq}  -\mu_{L}(E)$ to obtain:
$$\xymatrix@C15pt{ 0\ar[r]&\mathcal{O}_{X}(B)\ar[r] & E\otimes\mathcal{O}_{X}(B)\ar[r] & 
\mathcal{I}_{Y}\otimes \mathcal{L}\otimes\mathcal{O}_{X}(B)\ar[r]& 0,}$$
which induces an exact sequence on cohomology:
$$\xymatrix@C10pt{ 0\ar[r]&H^{0}(X,\mathcal{O}_{X}(B))\ar[r] & H^{0}(X,E\otimes\mathcal{O}_{X}(B))\ar[r] & H^{0}(X,\mathcal{I}_{Y}\otimes \mathcal{L}\otimes\mathcal{O}_{X}(B)).}$$
Assumption (i) implies $\delta_L(B)<0$, so  $H^{0}(X,\mathcal{O}_{X}(B)=0$.
As to the vanishing of the right-hand term, let  $B':=B+c_{1}(\mathcal{L})=B+c_{1}(E)$, so that $\delta_L(B')  \underset{(<)}{\leq} \mu_{L}(E)$.
By assumption (ii), there is no effective divisor of class $B'$ containing $Y$, and we claim therefore
$$H^{0}(X,\mathcal{I}_{Y}\otimes \mathcal{L}\otimes\mathcal{O}_{X}(B))=0.$$
Indeed, considering the exact sequence
$$\xymatrix{ 0\ar[r]&\mathcal{I}_{Y}\otimes \mathcal{L}\otimes\mathcal{O}_{X}(B)\ar[r] & \mathcal{L}\otimes\mathcal{O}_{X}(B)\ar[r] & \mathcal{L}\otimes\mathcal{O}_{X}(B)_{|Y}\ar[r]& 0,}$$
a global section of $\mathcal{I}_{Y}\otimes \mathcal{L}\otimes\mathcal{O}_{X}(B)$ provides a global section of $\mathcal{L}\otimes\mathcal{O}_{X}(B)$, which is trivial on $Y$.
\end{proof}

\begin{remark}
When the Picard rank of $X$ is large, there is a helpful technique to find necessary conditions on the class of an effective divisor $\mathcal{S}$ so that 
$$
\delta_L(\mathcal{S})  
\underset{(<)}{\leq} \mu_{L}(E).
$$
Suppose first that $X$ is a surface. Expressing  $\left[\mathcal{S}\right]$ in terms of the basis of $\Pic X$, we obtain constraints in inequality form. As an obvious starter we have:
\begin{equation}
\delta_L(\mathcal{S})>0.
\end{equation}
Now let $C$ be curve in $X$ such that $\delta_L(\mathcal{S})<\delta_L(C)$.
Then the following  sequence is exact:
\begin{equation}\label{es}
\xymatrix@C15pt@R0pt{ 0\ar[r]&H^{0}(X,\mathcal{O}_{X}(\mathcal{S}-C))\ar[r] & H^{0}(X,\mathcal{O}_{X}(\mathcal{S}))
\ar[r] & H^{0}(X,\mathcal{O}_{X}(\mathcal{S})_{|C}).}
\end{equation}
By assumption $H^{0}(X,\mathcal{O}_{X}(\mathcal{S}-C))=0$, and since $H^{0}(X,\mathcal{O}_{X}(\mathcal{S}))\neq0$ it follows that $H^{0}(X,\mathcal{O}_{X}(\mathcal{S})_{|C})\neq0$. Hence $\deg \mathcal{O}_{X}(\mathcal{S})_{|C}\geq 0$, so:
\begin{equation}
\mathcal{S}\cdot C\geq0.
\label{findS}
\end{equation}

The technique also works when $\dim X>2$. Let $C$ be an effective divisor in $X$ such that
\begin{itemize}
\item[(i)] $\delta_L(S) < \delta_L(C)$; 
\item[(ii)] there exists an ample class $\mathcal{A}$ of $C$ such that $\mathcal{A}=A|_C$ with $A$ being an ample class on $X$.
\end{itemize}
Then, as before, we have the exact sequence (\ref{es}), from which it follows that \linebreak $H^{0}(X,\mathcal{O}_{X}(\mathcal{S})_{|C})\ne0$. As a consequence, either $S|_C$ is an effective divisor on $C$ or it vanishes. Hence, by the Nakai--Moishezon criterion, $\mathcal{A}^{n-2}\cdot S|_C\ge0$, that is to say $S\cdot C\cdot A^{n-2}\ge0$.
\end{remark}

\subsection{The blow-up of $\pp^3$ along a plane conic}

We show applications of this adapted criterion by providing two asymptotically stable bundles on the particular Fano 3--fold listed no. 30 
in \cite{Mori}. One could apply the same method to obtain more examples over other building blocks.  

Let $r:X\rightarrow \mathbb{P}^{3}$ be the  blow-up of $\mathbb{P}^{3}$ along a plane conic $C$  
and $D\in \left|-K_{X}\right|$ be a  generic smooth K3 surface, corresponding to the proper transform of a quartic in $\mathbb{P}^{3}$ containing $C$. 
We denote by $H=r^{*}(H_{0})$ the pullback of the hyperplane  $H_{0}\subset\mathbb{P}^3$, by $\widetilde{C}$ the exceptional divisor, by $h=r^{*}(h_{0})$ the pullback of a line $h_{0}\subset\mathbb{P}^3$ and by $l$ a  fibre of $\widetilde{C}\rightarrow C$.

\begin{lemma}\label{Johannes}
The Picard lattice of $D$ is spanned by the hyperplane class $A$ and the intersection with the exceptional divisor, which is isomorphic to $C$ in $D$. The intersection form in this basis is 
$$\begin{pmatrix}
4&2     \\
2 & -2 \\
 \end{pmatrix}.$$
\end{lemma}
\begin{proof}
The curve $C$ has genus $0$ in $D$, hence $C^{2}=-2$.
Let $j:D\hookrightarrow \mathbb{P}^3$ be the inclusion map.
We have $A=j^{*}(H_{0})$,
hence $A^{2}=j^{*}(H_{0}^{2})=j^{*}(H_{0}^{2})\cdot D$.
Then the projection formula yields:
\begin{align*}
A^{2}&=H_{0}^{2}\cdot 4H_{0}=4.
\end{align*}
Similarly,  $A\cdot C= j^{*}(H_{0})\cdot C$ 
implies:
\begin{align*}
A\cdot C&=H_{0}\cdot 2h_{0}=2.\qedhere
\end{align*}
\end{proof}

In each of the following examples, we verify the hypotheses of Theorem \ref{thm: Hartshorne--Serre} for the existence of a Hartshorne--Serre bundle, then we check its asymptotic stability using Proposition \ref{Hartshorne} on its restriction to the K3 surface $D$.

\subsubsection{First example}

We apply Theorem \ref{thm: Hartshorne--Serre} to $X$ as above, with 
$$
Y=l \qand \mathcal{L}=\mathcal{O}_{X}(\widetilde{C}).
$$

\begin{proposition}\label{exa1}
Let $X\rightarrow \mathbb{P}^{3}$ be the  blow-up of $\mathbb{P}^{3}$ along
a plane conic $C$, $D\in \left|-K_{X}\right|$ be a  generic smooth K3 surface, corresponding
to the proper transform of a quartic in $\mathbb{P}^{3}$ containing $C$, and $l\subset X$ a fibre of the exceptional divisor $\widetilde{C}\rightarrow
C$; then there exists a rank 2 Hartshorne--Serre  bundle $E\to X$ obtained from $l$ such that:
\begin{itemize}
\item[(i)] $c_{1}(E)=\widetilde{C}$,
\item[(ii)] $c_{2}(E)=\left[l\right]$,
\item[(iii)] $E_{|D}$ is $A$-stable.
\end{itemize}
\end{proposition}

\noindent \uline{Existence of the Hartshorne--Serre bundle}

\begin{lemma}
In the hypotheses of Proposition \ref{exa1}, $H^{2}(X,\mathcal{O}_{X}(-\widetilde{C}))=0$.
\end{lemma}
\begin{proof}
We obtain, from the short exact sequence,
$$\xymatrix{ 0\ar[r]&\mathcal{O}_{X}(-\widetilde{C})\ar[r] & \mathcal{O}_{X}\ar[r] & \mathcal{O}_{\widetilde{C}}\ar[r]& 0}$$
 the following exact sequence on cohomology:
\begin{equation}
\xymatrix{H^{1}(\widetilde{C},\mathcal{O}_{\widetilde{C}}) \ar[r]&H^{2}(X,\mathcal{O}_{X}(-\widetilde{C}))\ar[r] & H^{2}(\widetilde{C},\mathcal{O}_{\widetilde{C}}).}
\label{H1C3}
\end{equation}
Since $\tilde{C}$ is a ruled surface, we have $H^{2}(\widetilde{C},\mathcal{O}_{\widetilde{C}})=0$, and
$H^{1}(\widetilde{C},\mathcal{O}_{\widetilde{C}})=H^{1}(C,\mathcal{O}_{C})$. However, $H^{1}(C,\mathcal{O}_{C})=0$ because
$C$ is a conic.
\end{proof}

\begin{lemma}   \label{lemma: normal bundle}
In the hypotheses of Proposition \ref{exa1}, $\mathcal{O}_{X}(\widetilde{C})_{|l}=\wedge^{2}\mathcal{N}_{l/X}$.
\end{lemma}
\begin{proof}

Since the claim concerns line bundles on a rational curve, it is enough to show that $c_1(\mathcal{O}_{X}(\widetilde{C})_{|l})=c_1(\wedge^{2}\mathcal{N}_{l/X})$. On the one hand, we have $c_{1}(\mathcal{O}_{X}(\widetilde{C})_{|l})=-1$, because $\widetilde{C}\cdot l=-1$. On the other hand, one can prove that $\wedge^{2}\mathcal{N}_{l/X}\simeq\mathcal{O}_l(-1)$ by a local argument, on a neighborhood of the line $l$; however, here goes a shorter proof using intersection theory in $X$.

By adjunction, we have 
\begin{eqnarray*}
        c_{1}(\wedge^{2}\mathcal{N}_{l/X})
        &=& 
        c_{1}(\mathcal{N}_{l/X})=c_{1}((\mathcal{T}_{X})_{|l})-c_{1}(\mathcal{T}_{l})
        \\
        &=&c_{1}(\mathcal{T}_{X})\cdot l-c_{1}(\mathcal{T}_{l}).
\end{eqnarray*}
We know that $c_{1}(\mathcal{T}_{X})=-c_{1}(\omega_{X})= 4H-\widetilde{C}$, so 
$$c_{1}((\mathcal{T}_{X})_{|l})=(4H-\widetilde{C})\cdot l=1.$$
Moreover, since $l$ is a line, we have $c_{1}(\omega_{l})=-2$, so 
$$
c_{1}(T_{l})=2,
$$
which yields $c_{1}( \wedge^{2}\mathcal{N}_{l/X})=-1$, as desired.
\end{proof}

\noindent \uline{Stability of $E|_D$}

Notice that  $E_{|D}$ is also a Hartshorne--Serre bundle obtained from the point $y:=l\cap D$, with $c_{1}(E_{|D})=C$.
We apply the adapted stability criterion from Proposition \ref{Hartshorne} with the polarisation $A$, for which $\mu_A(E)=\frac{A\cdot C}{2}=1$.
Since the intersection form on $D$ is even, there is no curve $\mathcal{S}$ such that $\delta_A(\mathcal{S}) \leq 1$, hence $E_{|D}$ is stable.

\subsubsection{Second example}
The next example is slightly more sophisticated. For $c_{2}(E)$ fixed, $A$--stability over $D$ depends on the choice of the generating subscheme $Y$, which in our case is a curve. We consider $\mathcal{R}$ the proper transform of a line $\mathcal{R}_0$ in $\mathbb{P}^{3}$ which meets the conic $C$ in just one point:
\begin{equation}
\left[\mathcal{R}\right]=h-\left[l\right].
\label{d}
\end{equation}
Let $\mathcal{P}$ be the plane that contains the conic $C$.
The intersection $\mathcal{P}\cap D$ in $\mathbb{P}^{3}$ is a plane curve of degree 4 containing $C$, hence it is a reduced curve having $C$ as an irreducible component. Since $D$ is generic, the only other irreducible component is another smooth plane conic, which we denote by $C^{\vee}$. 

\begin{definition} \label{R0 generic}
We say that the line $\mathcal{R}_0$ is \emph{generic} if $\mathcal{R}_0$ is not a totally tangent line to $D$ in $\mathbb{P}^{3}$ at a point of $C$, that is $D\cap \mathcal{R}_0$ does not consist of a single degree 4 point in $C$.
\end{definition}

We now apply Theorem \ref{thm: Hartshorne--Serre} with $Y=\mathcal{R}$ and $\mathcal{L}=H$.

\begin{proposition}\label{exa2}
Let $X\rightarrow \mathbb{P}^{3}$ be the  blow-up of $\mathbb{P}^{3}$ along
a plane conic $C$, $D\in \left|-K_{X}\right|$ be a  generic smooth K3 surface,
corresponding
to the proper transform of a quartic in $\mathbb{P}^{3}$ containing $C$,  $\mathcal{R}$ be the proper transform of a
line $\mathcal{R}_0$ in $\mathbb{P}^{3}$ which meets the conic $C$ in just one point and $H\subset X$ be the pullback of a hyperplane; then there exists a rank 2 Hartshorne--Serre bundle $E\to X$, obtained from $\mathcal{R}$, such that:
\begin{itemize}
\item[(i)] $c_{1}(E)=H$,
\item[(ii)] $c_{2}(E)=\left[\mathcal{R}\right]$,
\item[(iii)] $E_{|D}$ is $A$-stable, if $\mathcal{R}_0$ is generic,
and $A$-semi-stable otherwise.
\end{itemize}
\end{proposition}

\noindent \uline{Existence of the Hartshorne--Serre bundle}

Consider the exact sequence
$$ 0 \to \mathcal{O}_X(-H) \to \mathcal{O}_X \to \mathcal{O}_H \to 0 , $$
where $H$ can be regarded as the blow-up of a plane in two points. It follows that
\begin{equation}\label{H2}
H^{2}(X,\mathcal{O}_{X}(-H))=H^1(H,\mathcal{O}_H)=0.
\end{equation}
To use the Hartshorne--Serre construction, we need also the following lemma.
\begin{lemma}
We have $\mathcal{O}_{X}(H)_{|\mathcal{R}}=\wedge^{2}\mathcal{N}_{\mathcal{R}/X}$.
\end{lemma}
\begin{proof}
As in the proof of Lemma \ref{lemma: normal bundle}, it is enough to show that the Chern classes of both line bundles coincide.
On the one hand, we have  $c_{1}(\mathcal{O}_{X}(H)_{|\mathcal{R}})=1$, because
$H\cdot \mathcal{R}=1$.
On the other hand, since $c_{1}(\mathcal{T}_{X})= 4H-\widetilde{C}$, we have
$$c_{1}((\mathcal{T}_{X})_{|\mathcal{R}})=(4H-\widetilde{C})\cdot \mathcal{R}=3.$$
Since $\mathcal{R}$ is a line, $c_{1}(T_{\mathcal{R}})=2$.
It follows that $c_{1}( \wedge^{2}\mathcal{N}_{\mathcal{R}/X})= c_{1}((\mathcal{T}_{X})_{|\mathcal{R}}) - c_{1}(T_{\mathcal{R}})  = 1$.
\end{proof}

\noindent \uline{Stability of $E|_D$}

The bundle $E_{|D}$ is also a Hartshorne--Serre bundle obtained from the intersection $\xi:=\mathcal{R}\cap D$, which is a discrete set. Moreover $c_{1}(E_{|D})=A$.
Again by the stablity criterion of  Proposition \ref{Hartshorne}, with $\mu_A(E)=\frac{A\cdot A}{2}=2$,
in order to prove stability (resp. semi-stability) we have to check that $\xi$ is not in $\mathcal{S}$ for any curve $\mathcal{S}$ with $\delta_A(\mathcal{S}) \leq 2$ (resp. $\delta_A(\mathcal{S}) < 2$).
\\ For semi-stability, note that the intersection on $D$ is even, so the condition on $\mathcal{S}$ is actually  $\delta_A(\mathcal{S})\leq 0$. Thus
such  $\mathcal{S}$ cannot exist, because the degree of an effective divisor is always strictly positive. 

Now we address strict stability.
The bundle $E_{|D}$ is stable if $\xi$ is not contained in any curve $\mathcal{S}$ with $\delta_A(\mathcal{S})\leq 2$. Since $\delta_A(\mathcal{S})>0$ and the intersection on $D$ is even, it follows that $\delta_A(\mathcal{S})=2$.
Then, by projection formula, it follows that $\mathcal{S}$ is a curve of degree two in $\mathbb{P}^{3}$. 
Since $D$ is generic, there are just two possibilities to be ruled out: either $\mathcal{S}=C$, or $\mathcal{S}=C^{\vee}$. So, for the first case, we must argue that $\xi\not\subset C$; for the second case, we must argue that $\xi\not\subset C^{\vee}$.

If $\xi\subset C$, then $\mathcal{R}_0 \cap D \subset C$. Since $\mathcal{R}_0 \cap C$ consists of a single point, then $\mathcal{R}_0$ is a totally tangent line to $D$ at a point of $C$. However, this contradicts the hypothesis of $\mathcal{R}_0$ being generic in the sense of Definition \ref{R0 generic}.

If $\xi\not\subset C$ and $\xi\subset C^{\vee}$, then, since $\mathcal{R}_0$ is passing through one point in $C$, the line $\mathcal{R}_0$ lies within the plane $\mathcal{P}$; it follows that $\mathcal{R}_0\cap C$ has length two, which contradicts $\mathcal{R}_0\cap C$ consisting of a single point.

\section{Degeneration of asymptotically stable bundles}
\label{sec: degenerations}

We used monads to obtain examples of instanton bundles in \emph{Subsection \ref{subsec: stable bundles from monads}}. Another illustration of the usefulness of monads in gauge theory is the modelling of degenerating instanton sequences. Concretely, let us examine the task of producing  a one-parameter family $\left\{E_\lambda\right\}_{\lambda>0}$ of asymptotically stable locally-free sheaves over some building block, say $X=\pp^3$, such that $E_0:=\lim\limits_{\lambda\rightarrow 0}E_\lambda$ is torsion-free but not locally free. It seems reasonable to expect that this principle may in the future contribute to shed some light on the currently open and very hard problem of moduli space compactification for instantons in higher dimensions.

Recall from standard sheaf theory that the \emph{singular locus} of a linear monad (\ref{eq: linear monad}) over $X=\pp^n$ 
is the set 
\begin{displaymath}
\Sigma=\left\{z\in \pp^n \mid \ker\alpha(z) \neq \left\{ 0 \right\} \right\},
\end{displaymath}
where $\alpha(z)$ denotes the fibre map at the point $z\in \pp^n$ of the morphism $\alpha$ appearing in (\ref{eq: linear monad}). We know from \cite[Proposition 4]{Jardim2007} that the degeneration of the associated linear sheaf $E$ is essentially determined by the topology of $\Sigma$:

\begin{criterion}       \label{cri sheaves}
Let $E$ be a linear sheaf with singular locus $\Sigma$; then
\begin{enumerate}
        \item
        $E$ is locally-free $\Leftrightarrow$ $\Sigma=\varnothing$.
        \item
        $E$ is reflexive $\Leftrightarrow$ $\Sigma$ is a subvariety with $\codim \Sigma\geq3$.
        \item
        $E$ is torsion-free $\Leftrightarrow$ $\Sigma$ is a subvariety with $\codim \Sigma\geq2$.
\end{enumerate}
\end{criterion}

This result is relevant for the following Proposition.
 
\begin{proposition}
\label{prop: curvature blow-up}
Let $\left\{\rm M_\bullet(\lambda)\right\}_{\lambda>0}$ be the family of monads over $\pp^3$ 
$$
\xymatrix{
\rm M_\bullet(\lambda): & 0 \ar[r] & \OO_{\pp^3}(-1) \ar[r]^-{\alpha_\lambda} & \OO^{\oplus4}_{\pp^3} \ar[r]^-{\beta} & \OO_{\pp^3}(1) \ar[r] & 0
}
$$
given in homogeneous coordinates  $[z_1:z_2:z_3:z_4]$ by
\begin{displaymath}
\alpha_\lambda=
\left( \begin{array}{cccc}
z_1 & z_2 & \lambda z_3 & \lambda z_4 \\
\end{array} \right)^t
\qquad \text{and} \qquad
\beta=\left( \begin{array}{cccc}
-z_2 & z_1 & -z_4 & z_3 \\
\end{array} \right).
\end{displaymath}
Then the corresponding family of linear sheaves  $\left\{E_\lambda\right\}_{\lambda>0}$ has the following properties: 
\begin{enumerate}[(i)]
        \item
        each $E_\lambda$ is a  locally free sheaf
with   $\rk(E_\lambda)=2$, ${c_1(E_\lambda)=0}$ and $c_2(E_\lambda)=1$; 
        \item
        each $E_\lambda$ is asymptotically stable, i.e., each $E_\lambda|_D$ is
        stable over a fixed  anticanonical  quartic $D\in|\OO_{\pp^3}(4)|$; 
        \item
        The limit  $E_0:=\lim\limits_{\lambda\rightarrow 0}E_\lambda$ is a properly torsion-free (i.e. not locally-free),  asymptotically semistable sheaf with singular locus $\Sigma_0=\left\{ z_1=z_2=0 \right\}$;
        \item
        With respect to the family of standard Hermitian metrics induced from
$\mathbb{C}^4$ on the elements of  $\left\{E_\lambda\right\}$, the corresponding curvatures $F_\lambda$ blow up along $\Sigma_0$ at the rate $\frac{1}{\lambda}$ as $\lambda\to0$, in the sense that $\displaystyle\lim_{\lambda\to0}\left(\lambda^m\left\Vert F_\lambda\right\Vert_{L^1(\Sigma_0)}\right)=\infty$ if $m<1$, while $\lambda\left\Vert F_\lambda\right\Vert_{L^1(\Sigma_0)}$ is bounded for small $\lambda$.
\end{enumerate}
\end{proposition}
From the viewpoint of gauge theory outlined in the \emph{Introduction}, condition $(ii)$ implies, by \cite[Theorem 58]{SaEarp2011}, that each holomorphic bundle $E_\la|_{W}$, with $\la>0$, over the ACyl Calabi-Yau $3$--fold $W:=(\operatorname{Bl}_\ell \pp^3)\smallsetminus \tilde D$, admits a HYM connection $H_\la$ exponentially asymptotic to the ASD instanton on $E_\la|_D$ along the cylindrical end. In this context, however, the Hermitian metric on each $E_\la$ whose curvature $F_\la$ we study \emph{does not} a priori correspond to the actual HYM solution. Indeed,
such solutions in general are not known explicitly, and in the noncompact
ACylCY case not even the Kähler structure can be written out in coordinates,
let alone the HYM condition be explicitly verified.
We consider instead standard Hermitian metrics induced from
$\mathbb{C}^4$, expecting that the method outlined in the proof of Proposition  
\ref{prop: curvature blow-up} can be used to show, in the future, that HYM
solutions $H_\la$  display some similar degeneration patterns, hence inform
us about the structure of the tangent cone of the singular $\la\to0$ limit.

The rest of this Section is devoted to establishing Proposition  
\ref{prop: curvature blow-up}. Note that the above construction is really a proof of principle, indeed much more general topological types for $E_\lambda$ can be arranged.

\subsection{Properly torsion-free limit as $\lambda\to0$}

For $\lambda>0$, clearly $\Sigma_\lambda=\emptyset$, hence the corresponding linear sheaf is locally free.  It follows from Theorem \ref{Thm cyclic instanton monad } that $E_\lambda$ is stable, while Proposition \ref{Prop asymp stable middle cohomology} shows that $E_\lambda$ is asymptotically stable. This
proves claims  \emph{(i)} and \emph{(ii)} in  Proposition  
\ref{prop: curvature blow-up}. 

The limit $E_0$ is obviously still a linear sheaf. The new phenomenon  is that
\begin{displaymath}
\Sigma_0=\left\{ z_1=z_2=0 \right\}\subset\pp^3
\end{displaymath}
is a \emph{curve}, hence $E_0$ is a properly torsion-free sheaf. One can show that $E_0$ is properly semistable, see \cite[Proposition 14 and Example 4]{Jardim2006}. Its restriction to an anticanonical divisor $D$ is also properly torsion-free, since it intersects the singular locus at a single point.

Moreover, $E_0|_D$ is properly semistable; indeed, the anticanonical quartic
has $\Pic(D)=\Z$, so suffices to check that $h^0(\left.E_0(-1)\right\vert_D)=h^0(\left. E_0^*(-1)\right\vert_D)=0$, cf. \cite[Lemma 13]{Jardim2006}. We already know that  $h^0(E_0(-1)|_D)=0$. Recall that $K=\ker\beta$; since $E_0^*$ is a subsheaf of $K^*$, we conclude that  $h^0(K^*(-1)|_D)=0$ implies $h^0(\left.E_0^*(-1)\right\vert_D)=0$; to check that $h^0(\left.K^*(-1)\right\vert_D)=0$, simply consider the restriction sequence
$$ 
\xymatrix{
0 \ar[r] &  K^*(-5) \ar[r] &  K^*(-1) \ar[r] &  \left.K^*(-1)\right\vert_D \ar[r] &  0 
}$$
and note that $h^0( K^*(-1))=h^1(K^*(-5))=0$, which follows from the sequence

$$ 
\xymatrix{
0 \ar[r] &  {\mathcal{O}_{\mathbb{P}^3}}(-1) \ar[r]^{\beta^*} &  {\mathcal{O}^{\oplus 4}_{\mathbb{P}^3}} \ar[r] &  K^* \ar[r] &  0
}. 
$$ 
Therefore, we conclude that the limit sheaf $E_0$ is asymptotically  (properly) \linebreak semistable, which is claim \emph{(iii)} in Proposition \ref{prop: curvature blow-up}.

Note that there is nothing particular about $\pp^3$ here. Similar families of  monads can be constructed over a wide class of projective Fano 3-folds  $X$, say, using $\mc{O} _X(1)$ and the embedding coordinates to form the maps.
\subsection{Curvature blow-up along degenerating instanton sequences}

An interesting feature of such explicit models is the quantitative description of the curvature blow-up rate as the family degenerates into the limiting torsion-free sheaf.   

In order to check claim \emph{(iv)} in Proposition \ref{prop: curvature blow-up},
we set out to write explicitly the natural curvature associated to the holomorphic structure of a linear sheaf $E_\la$. Dualising the map  $\alpha_\la $ in $M_\bullet(\la)$ we have an `Euler characteristic' map
\begin{displaymath}
        R_\la:=\beta\oplus\alpha_\la^\vee:\OO_{\mathbb{P}^3}^{\oplus4}\to\OO_{\mathbb{P}^3}(1)^{\oplus2}
\end{displaymath}
given  by
\begin{displaymath}
        R_\la=
        \begin{pmatrix}
                \bar z_1 & \bar z_2 & \lambda\bar z_3 & \lambda\bar z_4 \\
                -z_2 & z_1 & -z_4 & z_3 \\
        \end{pmatrix}
        \quad\text{thus}\quad
        R_\la^\vee=
        \begin{pmatrix}
                z_1 & -\bar z_2 \\
                z_2 &\bar z_1   \\
                \lambda z_3 & -\bar z_4 \\
                \lambda z_4 & \bar z_3  \\
        \end{pmatrix}.
\end{displaymath}
Let $P_\la: \Gamma(\OO_{\mathbb{P}^3}^{\oplus4}) \to \Gamma(E_\la)$ denote the orthogonal projection onto $\ker R_\la$; this is given by
$$ P_\la =1-R_\la^\vee D_\la R_\la, $$
with
$$ \begin{array}{rcrcl}
D_\la &:=& \left(R_\la R_\la^\vee\right)^{-1}&=&
\begin{pmatrix}
        \displaystyle\frac{1}{\left\vert \alpha_\la \right\vert^2} & 0 \\
        0 & \displaystyle\frac{1}{\left\vert \beta \right\vert^2} \\
\end{pmatrix}\\
\end{array}. $$

The crucial observation is that, fixing a Hermitian metric (e.g. the standard
metric from $\mathbb{C}^4$), the middle cohomology bundle can be identified as
$$
E_\la=\frac{\ker \beta}{\im \alpha_\la}\simeq\ker R_\la
$$
and the induced Chern connection is given by $\nabla_\la= P_\la \circ d$:
$$
\xymatrix{
        \Gamma(E_\la) \ar[r]^-{\nabla_\la}\ar[d]_-{\iota} & \Gamma(E_\la)\otimes\Omega^1 \\
\Gamma(\OO^{\oplus4}) \ar[r]^-{d} & \Gamma(\OO^{\oplus4})\otimes\Omega^1 \ar[u]_{ P_\la } \\
}
$$
so curvature is given by $F_\la=( P_\la d)\circ( P_\la d)=(P_{\la}\circ dP_\la)\circ d$.

We will now determine explicitly the  curvature along the curve $\Sigma_0$.
By direct differentiation, one has 
\begin{eqnarray}        \label{eq: dP}
        -dP_\la &=& \left(dR_\la^\vee\right) D_\la R_\la
                   +R_\la^\vee \left[D_\la \circ dR_\la 
                   +\left(dD_\la\right) \circ R_\la\right].
\end{eqnarray}
Specialising to the singular locus $\Sigma_0=\left\{ z_1=z_2=0 \right\}$,
with $(z_3,z_4)\neq(0,0)$, we have:
$$
R_\la^\vee|_{\Sigma_0}= 
\begin{pmatrix}
        0 & 0 \\
        0 & 0 \\
        \lambda z_3 & -\bar z_4 \\
        \lambda z_4 & \bar z_3  \\
\end{pmatrix}
\qand
D_\la|_{\Sigma_0}
=
\frac{1}{\left\vert z_3\right\vert^2+\left\vert z_4\right\vert^2}
\begin{pmatrix}
        \frac{1}{\lambda^2} & 0 \\
        0 & 1 \\
\end{pmatrix},
$$
so the projection restricted to $\Sigma_0$ takes the upper-left  $2\times2$ block diagonal form:
$$
\begin{array}{ccccc}
P_\la|_{\Sigma_0}
&=&1-\frac{1}{\left\vert z_3\right\vert^2+\left\vert z_4\right\vert^2}
R_\la^\vee|_{\Sigma_0} 
\begin{pmatrix}
        \frac{1}{\lambda^2} & 0 \\
        0 & 1 \\
\end{pmatrix}
R_\la|_{\Sigma_0} 
&=&\begin{pmatrix}
        1 &  &  & \\
         & 1 &  & \\
         &  & 0 & \\
         &  &  & 0\\
\end{pmatrix}.
\end{array}
$$
In particular,  
$\left(P_\la \circ R_\la^\vee\right)|_{\Sigma_0}=0$, so the only non-trivial contribution to 
$\left(P_\la \circ dP_\la\right)|_{\Sigma_0}$ comes from the first summand in (\ref{eq: dP}). Defining
for convenience 
$$
\xi_{ij}:=\bar z_idz_j,
$$ 
we have
\begin{eqnarray*}
\left[\left(dR_\la^\vee\right) D_\la R_\la\right]|_{\Sigma_0}
&=&
\displaystyle\frac{1}{\left\vert z_3\right\vert^2+\left\vert z_4\right\vert^2}
\begin{pmatrix}
        dz_{1} & -d\bar z_2\\
        dz_2 & d\bar z_1 \\
        \lambda dz_3 & -d\bar z_4 \\
        \lambda dz_4 & d\bar z_3  \\
\end{pmatrix}
\begin{pmatrix}
        \frac{1}{\lambda^2} & 0 \\
        0 & 1 \\
\end{pmatrix}
\begin{pmatrix}
        0 & 0 & \lambda\bar z_3 & \lambda\bar z_4 \\
        0 & 0 & -z_4 & z_3 \\
\end{pmatrix}\\
&=&
\displaystyle\frac{1}{\left\vert z_3\right\vert^2+\left\vert z_4\right\vert^2}
\left[
\begin{pmatrix}
        0 &  & \bar\xi_{42} & -\bar\xi_{32} \\
          & 0 & -\bar\xi_{41} & \bar\xi_{31}  \\
         &  & \xi_{33}  & \xi_{43} \\
        &  & \xi_{34}  & \xi_{44} 
\end{pmatrix}
+\frac{1}{\la}
\begin{pmatrix}
        0 &  & \xi_{31} & \xi_{41} \\
          & 0 & \xi_{32} & \xi_{42} \\
         &  &  \bar\xi_{44} & -\bar\xi_{34} \\
        &  &  - \bar \xi_{43} & \bar\xi_{33} 
\end{pmatrix}
\right].
\end{eqnarray*}

We conclude that
\begin{eqnarray*}
        -\left(P_\la \circ dP_\la\right)|_{\Sigma_0} &=&
        \frac{1}{\left\vert z_3\right\vert^2+\left\vert z_4\right\vert^2}
        \left[
        \begin{pmatrix}
                0  &  & \bar\xi_{42} & -\bar\xi_{32} \\
                   & 0&-\bar\xi_{41} & \bar\xi_{31}  \\
                   &  & 0 &  \\
                   &  &   & 0 \\
        \end{pmatrix}
        +
        \frac{1}{\la}
        \begin{pmatrix}
                0  &  & \xi_{31} & \xi_{41} \\
                   & 0& \xi_{32} & \xi_{42} \\
                   &  & 0&  \\
                   &  &  & 0\\
        \end{pmatrix} 
        \right]
\end{eqnarray*} 
with $\xi_{ij}$ not all zero.
Therefore, towards the limiting torsion-free sheaf $E_0$, curvature blows up along the curve $\Sigma_0$ as
$$
\left\vert\left.F_\la\right\vert_{\Sigma_0}\right\vert=\left\vert(\left.P_\la\right\vert_{\Sigma_0}\circ d\left.P_\la\right\vert_{\Sigma_0}\circ d)\right\vert\overset{\lambda\to0}{\simeq}\frac{1}{\lambda}\to\infty.
$$


\end{document}